\documentclass[reqno, 12pt, final]{amsart}

\usepackage[utf8]{inputenc}
\usepackage{amsaddr}
\usepackage{amsmath,amssymb,amsfonts,amsthm}

\usepackage{a4wide}

\usepackage{graphicx}
\usepackage[pdfpagelabels,hypertexnames=true,
            plainpages=false,
            naturalnames=false]{hyperref}
\usepackage{color}
\definecolor{darkblue}{rgb}{0,0.1,0.5}
\hypersetup{colorlinks,
            linkcolor=darkblue,
            anchorcolor=darkblue,
            citecolor=darkblue}

\usepackage{txfonts, bbm}
\usepackage{enumerate}
\usepackage{booktabs, slashbox}
\usepackage{caption}
\usepackage{subcaption}
\usepackage{mathrsfs}

\allowdisplaybreaks

\newtheorem{theorem}{Theorem}
\newtheorem{lemma}{Lemma}

\newtheorem{proposition}{Proposition}
\newtheorem{corollary}{Corollary}
\newcommand{\1}{\boldsymbol{1}}
\newcommand{\ii}{\mathrm{i}}
\newcommand{\dd}{\mathbbm{d}}
\newcommand{\bx}{\vec{x}}

\newcommand{\bxi}{\vec{\xi}}

\newcommand{\fitMC}{\vec{f}}

\renewcommand{\vec}{\mathbf}

\renewcommand{\leq}{\leqslant}
\renewcommand{\geq}{\geqslant}

\usepackage[capitalise,noabbrev]{cleveref}

\begin{document}
\title[A DE for the asymptotic fitness distribution in the B--S model with five species]{A differential equation for the asymptotic fitness distribution in the Bak--Sneppen model with five species}
\author{Eckhard Schlemm}

\address{University College London Medical School, Gower Street, London}
\email{eckhard.schlemm.13@ucl.ac.uk}
% 37N25: Dynamical systems in biology
% 92C42: Systems biology, networks 
% 60J05: Discrete-time Markov processes on general state spaces
% 62E20: Asymptotic distribution theory

\begin{abstract}
The Bak--Sneppen model is an abstract representation of a biological system that evolves according to the Darwinian principles of random mutation and selection. The species in the system are characterized by a numerical fitness value between zero and one. We show that in the case of five species the steady-state fitness distribution can be obtained as a solution to a linear differential equation of order five with hypergeometric coefficients. Similar representations for the asymptotic fitness distribution in larger systems may help pave the way towards a resolution of the question of whether or not, in the limit of infinitely many species, the fitness is asymptotically uniformly distributed on the interval $[f_c,1]$ with $f_c\gtrapprox 2/3$.
\end{abstract}

\maketitle

\section{Introduction}
The Bak--Sneppen (B--S) model is an abstract representation of a biological system that evolves according to the Darwinian principles of random mutation and natural selection. It was introduced in \cite{Bak1993punctuated} in the context of self-organized criticality in systems with spatial interactions. 

Despite its simplicity, the B--S model captures some of the features that are believed to be characteristic of evolving biological systems. In particular, it predicts evolutionary activity on all time scales with long periods of relative stasis interrupted by bursts of activities, referred to as \emph{avalanches}. As a consequence of the absence of a characteristic time scale, evolutionary dynamics in the B--S model display long-range dependence in both the temporal and the spatial domain.  It is thus suitable as an abstract representation of systems in punctuated equilibrium, a concept that was introduced in \cite{gould1972punctuated} to explain the patterns observed in fossil records. For a more thorough discussion of these ideas, and for applications of the notion of punctuated equilibrium in other scientific disciplines we refer the reader to \cite{jensen1988self}. 

In addition to its usefulness in abstractly representing some key features of palaeontology and macro-evolution, the B--S model has also been employed to analyze the evolution of bacteria in a controlled, competitive environment. In a series of key experiments \cite{lenski1994dynamics}, Lenski and collaborators cultivated twelve initially identical populations of an E.\ coli strain over several years and conserved samples at regular time intervals. They then determined the relative fitness of the conserved samples by putting them into direct competition with a sample taken from the initial populations and measuring their relative growth rates. In \cite{donangelo2002model} it was shown that the B--S model with random mutations qualitatively reproduces some of the experimental results on relative bacterial fitness obtained in Lenski's long-term experimental evolution project \cite{lenski2015LTEE}. Using an extended multi-trait variant, \cite{bose2001bacterial} extended the explanatory scope of the B--S model to include experimental findings about the interplay of adaptation, randomness and history in bacterial evolution.

\subsection*{Informal definition of the B--S model} The Bak--Sneppen model characterizes each species in a biological system by a numerical fitness value between zero and one, which represents its degree of adaptedness to its environment and changes as the species evolves. Further, each species is assumed to directly interact with exactly two other species, where it is left unspecified if such an interaction represents competition for resources, predator-prey relations or something different entirely. The Bak--Sneppen model can therefore be visualized as points on a circle, where each point stands for an ecological niche (or a species occupying that niche) and neighboring species interact with each other.

As time progresses, the fitness landscape evolves in accordance with the following rules, representing in an abstract way the principles of random mutation and natural selection: at each time step, the least adapted species, i.\,e.\, the one with the smallest fitness parameter, is removed from the system (becomes extinct) and its place is immediately taken by a new species whose fitness is initially modelled as an independent uniformly distributed random variable.  In order to take into account the effect of this change on the local environment, the fitness parameters of the two species to either side of the least fit one are also reset to random values. This can be thought of as those two species themselves becoming extinct and superseded by new ones, or as them undergoing mutations in response to their neighbour becoming extinct.

\subsection*{Previous mathematical results}Despite its apparently easy definition, the B--S model has withstood most attempts at mathematical analysis in the past. Partial results have been obtained, however, in the context of rank-driven processes \cite{grinfeld2011bak,grinfeld2011rank} and mean-field approximations \cite{de1994simple,flyvbjerg1993mean}. Early on it was conjectured based on simulations that the steady-state fitness distribution at a fixed site converges, in the limit of large populations, to a uniform distribution on the interval $[f_c,1]$, where $f_c$ is approximately equal to 0.667, but believed to be slightly larger than $2/3$.

There is only a small number of mathematically rigorous result about the Bak--Sneppen model; in \cite{MeesterZnamenski2003limit} it is shown that the steady-state fitness at a fixed site is bounded away from one in expectation, independent of the number of species in the system;  A characterization of the limiting marginal fitness distribution, conditional on a set of critical thresholds, is given in \cite{MeesterZnamenski2004critical} (see also \cite{gillet2006bounds,gillet2006maximal}). In \cite{schlemm2012asymptotic}, the author proposes to compute the steady-state fitness distribution as the fixed point of the one-step transition equation and uses this method to describe the asymptotic fitness distribution for four species in terms of a compact rational function. In the same paper it is shown that one cannot find a similarly simple formula in the B--S model with five species, and that the fitness distribution of a randomly selected species at steady-state in this case is not only not rational, but not even a hypergeometric function.

\subsection*{Our contribution}In this paper, we revisit the Bak--Sneppen model with five species. In \cref{thm-main}, our main theorem, we establish a representation of the steady-state fitness distribution for five species in terms of the solution of an explicit differential equation with hypergeometric coefficients. This steady-state distribution encodes information about the fitness attributes of species in a system that has evolved for a long time. For instance, one can deduce from it how fit, on average, a randomly selected species from the population is expected to be; this is done in \cref{cor-marg}. Furthermore, since the steady-state distribution contains information about the joint fitness values of all species in the population, its knowledge allows to draw biologically relevant conclusions about qualitative properties of the system, such as the emergence of one or several dominant species, or the fragmentation of the eco-system into areas of different prevailing fitness. In our simple model, the symmetry of the initial configuration is preserved and no such phenomenon occurs. It is an interesting question whether in systems with a more complicated interaction between species, symmetry can be spontaneously broken.

We envisage that similar representations for the asymptotic fitness distribution in larger systems may help pave the way towards a resolution of the question of whether or not, in the limit of infinitely many species, the fitness is indeed asymptotically uniformly distributed on the interval $[f_c,1]$ with $f_c\gtrapprox 2/3$. We speculate that the techniques developed in this paper can be generalized to analyze larger systems with more than five species. It seems plausible to expect that the asymptotic joint fitness distribution in such systems can still be characterized as the solution to a certain linear differential equation, even though the coefficients might no longer be hypergeometric functions.

\section{Formalization and main result}
We adopt the following formalization of the Bak--Sneppen model from \cite{schlemm2012asymptotic}. Initially, all fitness parameters are independent uniformly distributed and after $k$ evolutionary steps the state of the system is represented by the vector $\fitMC_k\in[0,1]^5$, where the $i$th component refers to the fitness of the $i$th species. The evolutionary dynamics of the system can be expressed formally by the equation
\begin{equation*}
	\mathbb{P}\left(\vec f_{k+1}\in A\left|\vec f_k=\bx\right.\right) = \int_A{\mathbb{P}_{\bx}(\dd^5\bxi)},\quad \bx\in[0,1]^5,\quad A\in\mathscr{B}([0,1]^5),
\end{equation*}
where the one-step transition kernel $\mathbb{P}_{\bx}$ encodes the dynamics of the model and is given by
\begin{equation*}
	\mathbb{P}_{\bx}(\dd^5\bxi) = \prod_{\mu\notin\{\nu-1,\nu,\nu+1\}}{\delta_{x_\mu}(\dd\xi_\mu)\dd^3(\xi_{\nu-1},\xi_{\nu},\xi_{\nu+1})},\quad \nu=\operatorname{argmin}{\bxi}.
\end{equation*}
Here, and in the following, all vector indices are taken modulo five. The sequence $\fitMC=(\fitMC_k)_k$ is a uniformly ergodic Markov chain with absolutely continuous marginal distributions with densities $g_k:[0,1]^5\to\mathbb{R}^+$. This means that for any Borel set $A\in\mathscr{B}([0,1]^5)$,
\begin{equation*}
	\mathbb{P}\left(\vec f_k\in A\right) = \int_{A}{g_k(\bx)\dd^5 \bx},
\end{equation*}
and that the random vectors $\vec f_k$ converge in distribution to a steady-state limit $\vec f_\infty$. Moreover, the $k$-step densities $g_k$ satisfy the recursion
\begin{equation}
\label{eq-recgk}
g_{k+1}(\bx)=\sum_{\nu=1}^5\int_{[0,1]^3} \1_{\left\{\xi_2<\min(\xi_1,\xi_3,\bx_{]\nu[})\right\}}g_k\left(\bx_{]\nu[_{\bxi}}\right)\dd^3\bxi,
\end{equation}
where the vectors $\bx_{]\nu[}\in[0,1]^2$ and $\bx_{]\nu[_{\bxi}}\in[0,1]^5$ are obtained from $\bx$ by dropping the $\nu$th, and $(\nu\pm1)$th components, or replacing these components by the components of $\bxi$, respectively. Uniform ergodicity of the Markov chain $\fitMC$ implies that the densities $g_k$ converge uniformly to the density $g=g_\infty$ of the unique invariant distribution of $\fitMC$, which we recognize as the steady-state fitness distribution. We also introduce the notation 
\begin{equation*}
	F_{n,m}(x) = {}_2\operatorname{F}_1\left\{\frac13\left(n+\ii\sqrt{2}\right),\frac13\left(n-\ii\sqrt{2}\right); \frac m3; x\right\},\quad n,m\in\mathbb{Z},\quad x\in\mathbb{R},
\end{equation*}
where ${}_2\operatorname{F}_1$ denotes the Gauss hypergeometric function \cite[Section 15.1]{abramowitz1964handbook} and $\ii=\sqrt{-1}$ is the imaginary unit. The following is the main result of the paper. 
\begin{theorem}
\label{thm-main}
The limiting density $g=\lim_{k\to\infty}g_k$ is given by 
\begin{equation}
\label{eq-limdensity}
g(\bx)=\mathbf{1}_{[0,1]^5}(\bx)\sum_{\nu=1}^5{q\left(\min{\{x_\nu,x_{\nu+1}\}},\max{\{x_\nu,x_{\nu+1}\}}\right)}, 
\end{equation}
where $q(x,y) = \mathcal{G}'(1-x)\mathcal{B}'_1(1-y)+B_{\circ,0}(x)$. Here,
\begin{equation}
\label{eq-defG}\mathcal{G}(x) = \frac{3}{2}F_{2,1}\left\{1/2\right\}F_{1,2}\left\{x^3/2\right\}+\frac98x F_{4,5}\left\{1/2\right\}F_{2,4}\left\{x^3/2\right\},
\end{equation}
and the function $\mathcal{B}_1$ is the unique solution of the differential equation $\sum_{j=0}^5{c_j(y)\mathcal{B}_1^{(j)}(y)}= 0$ with boundary conditions
\begin{equation}
\mathcal{B}_1(1)=1/5,\quad \mathcal{B}_1'(1) = 0,\quad \mathcal{B}_1''(1)=-1/5, \quad\mathcal{B}_1^{(3)}(1)= 1, \quad\mathcal{B}_1^{(4)}(1)= -18/5.%, \quad\mathcal{B}_1^{(5)}(1)= -6/5.
\end{equation}
The coefficients $c_j(y)$, $j=0,1,\ldots,5$, are hypergeometric functions given by
\begin{subequations}
\label{eq-coeff-ODE}
\begin{align}
c_0(y) =& \frac{18 y^4}{\left(y^3+2\right)^2}\left[y \left(y^3-22\right) \mathcal{G}'(y)+\left(5 y^3-14\right) \mathcal{G}(y)\right],\\
c_1(y) =& -y c_0(y),\\
c_2(y) =& \frac{6 }{y^3+2}\left[y \left(3 y^6-38 y^3-4\right) \mathcal{G}'(y)+\left(15 y^6-10 y^3+4\right) \mathcal{G}(y)\right],\\
c_3(y) =& -12 y \left[y \left(4 y^3-1\right) \mathcal{G}'(y)+\left(5 y^3+1\right) \mathcal{G}(y)\right],\\
c_4(y) =& -3 y^2 \left[y \left(y^3+2\right) \mathcal{G}'(y)+\left(9 y^3-2\right) \mathcal{G}(y)\right],\\
c_5(y) =& y^3 \left(y^3+2\right) \left[y \mathcal{G}'(y)-\mathcal{G}(y)\right].
\end{align}
\end{subequations}
Finally, the function $B_{\circ,0}$ is given by
\begin{equation}
\label{eq-defBcirc0}
B_{\circ,0}(x) = \int_{1-x}^1{\frac1\xi\left[\mathcal{G}''(\xi)\mathcal{B}_1(\xi) - \mathcal{G}(\xi)\mathcal{B}_1''(\xi)\right]\dd\xi}.
\end{equation}
\end{theorem}
Unfortunately, it does not seem possible to evaluate the differential equation $\sum_{j=0}^5{c_j(y)\mathcal{B}_1^{(j)}(y)}= 0$ in terms of known special functions. As a direct consequence of \cref{thm-main}, we obtain the fitness distribution of a single species at steady-state by computing the one-dimensional marginal of \cref{eq-limdensity}.
\begin{corollary}
\label{cor-marg}
In the Bak--Sneppen model with five species, the fitness distribution of a randomly-selected species at steady-state is absolutely continuous with density
\begin{equation}
\label{eq-marginal}
g_{\textrm{marg}}(x) = \left[\frac35 + \mathcal{B}_1'(1-x)\right]\mathbf{1}_{[0,1]}(x).
\end{equation}
\end{corollary}
\begin{proof}
	Without loss of generality we can compute the distribution of the fitness of a fixed species, say the first one, which we call $f_{\infty,1}$. Averaging over the fitness values of the remaining species and using \cref{eq-limdensity} leads to the expression
\begin{equation*}
g_{\textrm{marg}}(x) = \int_{[0,1]^4}{g(x,\xi_2,\ldots,\xi_5)\dd^4\xi}=\left(\frac35+2\int_0^x{q(y,x)\dd y}+2\int_x^1{q(x,y)\dd y}\right)\mathbf{1}_{[0,1]}(x)
\end{equation*}
for the density of $f_{\infty,1}$. Plugging in the explicit formula for $q$ stated in \cref{thm-main}, differentiating once with respect to $x$ and using \cref{eq-defBcirc0} to eliminate $B_{\circ,0}$ results in the equation $g_{\textrm{marg}}'(x) = \mathcal{B}_1''(x)$. The observation that $g_{\textrm{marg}}(0)=3/5$ together with $\mathcal{B}_1(1)=0$ completes the proof.
\end{proof}
The graphs of the limiting density $g_{\textrm{marg}}$ as well as the marginals of the $k$-step densities $g_k$ are depicted in \cref{fig-margdens}, illustrating the convergence asserted in \cref{thm-main}. Moreover, \cref{fig-comparemargdens} compares the cumulative distribution functions of the steady-state fitness value at a fixed site in die Bak--Sneppen model with three (trivial), four \cite[Theorem 1]{schlemm2012asymptotic} and five (\cref{thm-main}) species, as well as their conjectured limit -- assuming $f_c=2/3$ -- in a system with an infinite number of species. Clearly, a lot of work remains to be done to fill the gap between five and infinitely many species.

\begin{figure}
	\begin{subfigure}[t]{0.45\textwidth}\centering		
	\includegraphics[width=\textwidth]{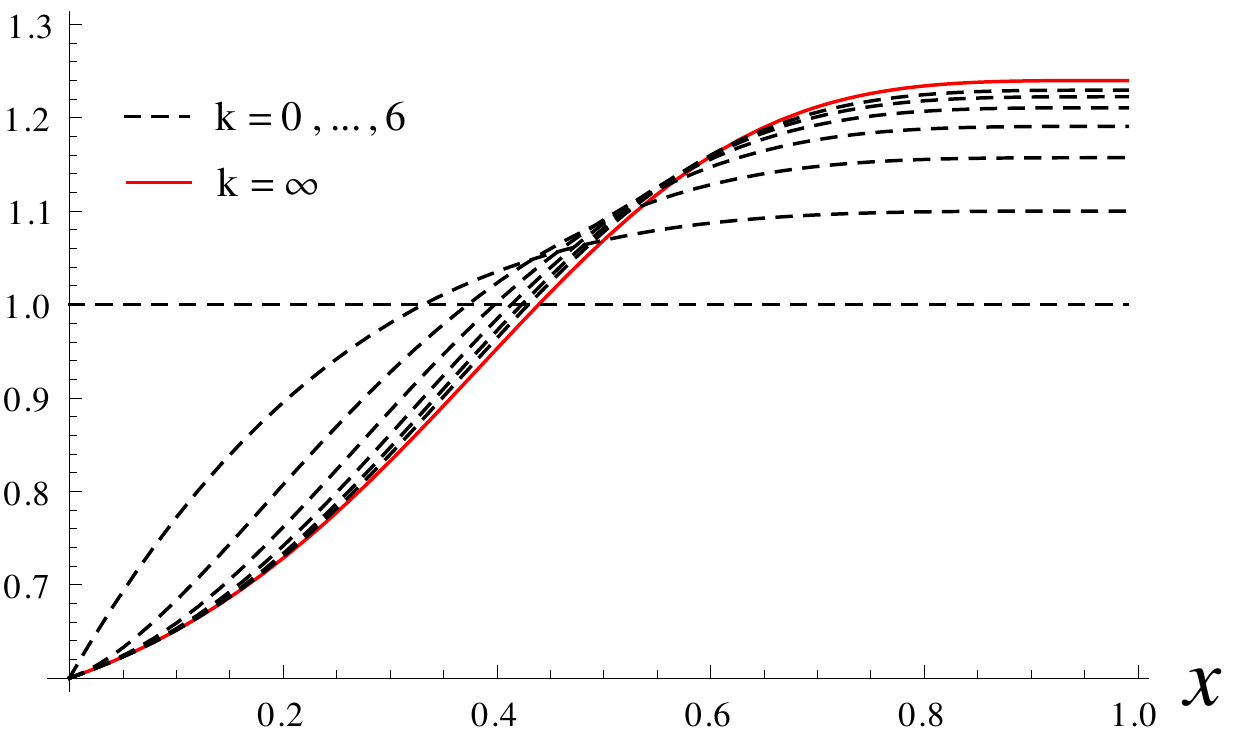}
\caption{Plot of the densities of the one-dimensional marginal distributions of $\fitMC_k$ for $k=0,\ldots,6$ (dashed), together with their limit (solid line) as given by \cref{eq-marginal}.}
\label{fig-margdens}
\end{subfigure}
\hspace{.05\linewidth}
\begin{subfigure}[t]{0.45\textwidth}\centering
\center
	\includegraphics[width=\textwidth]{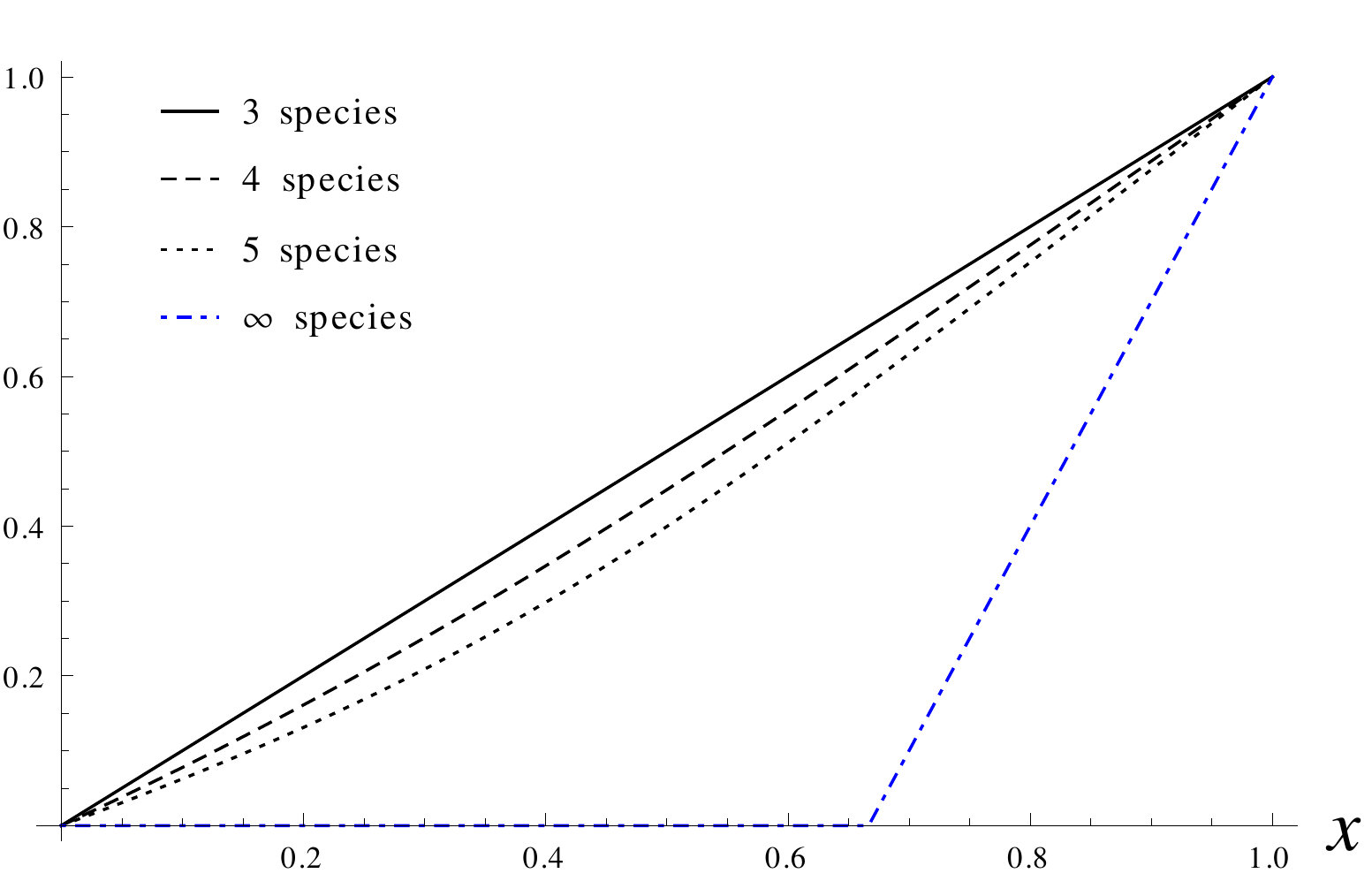}
\caption{Plot of the cumulative distribution functions of the marginal steady-state fitness distribution in the Bak--Sneppen model with three (solid), four (dashed) and five (dotted) species, as well as their conjectured limit as the number of species goes to infinity (dash-dotted line).}
\label{fig-comparemargdens}
\end{subfigure}
\end{figure}

\section{Proof of the main theorem}
It was shown in \cite[Proposition 2]{schlemm2012asymptotic} that the density of the joint distribution of the fitness parameters after $k$ steps, starting from a uniform distribution $\mathcal{U}([0,1])^{\otimes 5}$, is a polynomial given by
\begin{align*}
g_k(\bx) = \left[\sum_{\nu=1}^5{q_k\left(\min{\{x_\nu,x_{\nu+1}\}},\max{\{x_\nu,x_{\nu+1}\}}\right)}\right],
\end{align*}
where the functions $q_k\in\mathbb{Q}[x,y]$ can be written as
\begin{equation*}
q_k(x,y) = \sum_{i,j\geq 0}\alpha_{i,j,k}\,x^i\,y^j.
\end{equation*}
%The densities $g_k$, and thus implicitly the functions $q_k$, are defined recursively via
%This equation corresponds to the Mathematica code in \cref{listing-Ma-FokkerPlanck}. 
%\Cref{fig-margdens} shows a plot of the one-dimensional marginal densities of the fitness landscape $\fitMC_k$ at times $k=0,1,\ldots,5$ together with their limit as obtained from \cref{}. 
The coefficients of $q_k$ for $k=1,\ldots,5$, as computed from \cref{eq-recgk}, are tabulated in \cref{tab-coeff-k1,tab-coeff-k2,tab-coeff-k3,tab-coeff-k4,tab-coeff-k5}.
%\begin{lstlisting}[language=Mathematica,caption={Example code},label={listing-MMa-FokkerPlanck}]
%f[0, x1_, x2_, x3_, x4_, x5_] = 1;
%
%For[k = 1, k <= 2, k = k + 1,
% g[k, x1_, x2_] =  Integrate[Boole[y2 <= Min[y1, x1, y3, x2]] f[k - 1, x1, x2, y1, y2, y3], {y3, 0, 1}, {y2, 0, 1}, {y1, 0, 1}] // Simplify;
% 
% f[k, x1_, x2_, x3_, x4_, x5_] = Fold[#1 + (g[k, Sequence @@ #2]) &, 0, {{x1, x2}, {x2, x3}, {x3, x4}, {x4, x5}, {x5, x1}}];
% 
% m[k, x_] = 3/5 + Integrate[g[k, x1, x], {x1, 0, 1}] + Integrate[g[k, x, x1], {x1, 0, 1}] // Simplify;
% G[k, x_] = Integrate[m[k, y], {y, 0, x}];
% ]
%\end{lstlisting}
In order to identify the function $g$ featuring in \cref{thm-main}, it is thus sufficient to compute the uniform limit
\begin{equation}
\label{eq-def-q-as-lim}
q(x,y)\coloneqq\lim_{k\to\infty}q_k(x,y)=\lim_{k\to\infty}\sum_{i,j}{\alpha_{i,j,k}x^iy^j}. 
\end{equation}
This will occupy most of the rest of this section. First we recall an explicit recursion for the coefficients $\alpha_{i,j,k}$ that was derived in \cite[Proposition 5]{schlemm2012asymptotic}. It corresponds to -- and is derived from -- the recursion \labelcref{eq-recgk} for the densities $g_k$. It obviates the need to evaluate any integrals and thus allows for the functions $g_k$ to be determined much more quickly and efficiently. 
\begin{proposition}
\label{prop-rec-alpha5}
The coefficients $\alpha_{i,j,k}$ vanish for $i=0$ and have the following properties:
\begin{enumerate}[i)]
%\item $\alpha_{0,j,k}=0$ for all $k\geq 1$ and $j\geq 0$.
\item For $j=0$, they satisfy $\alpha_{1,0,k}=0$, $\alpha_{2,0,k+1}=2\sum_{p=0}^{3k+1}\frac{\alpha_{1,p,k}}{p+1}$, as well as
\begin{equation}
\begin{split}
\label{eq-recn5j0}
\alpha_{i,0,k+1} =& \alpha_{i-1,0,k}-\left[1+\frac{1}{(i-1)i}\right]\alpha_{i-2,0,k}+\left[\frac{1}{3}+\frac{1}{(i-2)i}\right]\alpha_{i-3,0,k}\\
&\quad+\frac{i+2}{i}\sum_{p=0}^{3k+1}\frac{\alpha_{i-1,p,k}}{p+1}-\frac{i+4}{2i}\sum_{p=0}^{3k+1}\frac{\alpha_{i-2,p,k}}{p+1}-\frac{i+2}{i}\sum_{p=0}^{i-2}\frac{\alpha_{i-2-p,p,k}}{p+1}\\
&\quad+\frac{i+4}{2i}\sum_{p=0}^{i-3}\frac{\alpha_{i-3-p,p,k}}{p+1}+\sum_{p=0}^{i-2}\frac{\alpha_{i-2-p,p,k}}{i-p} -\frac{1}{2} \sum_{p=0}^{i-3}\frac{\alpha_{i-3-p,p,k}}{i-p},\quad i\geq3.
\end{split}
\end{equation}
\item For $j\geq 1$, they satisfy the recursion
\begin{equation}
\label{eq-recn5jgeq1}
\alpha_{i,j,k+1} = \begin{cases}
\sum_{p=0}^{3k+1}\frac{\alpha_{j,p,k}}{p+1}+\sum_{p=0}^{j-1}\frac{2p-j+1}{(p+1)(j-p)}\alpha_{j-1-p,p,k},		      & i=1,\\
\alpha_{1,j,k}-\frac{1}{2}\alpha_{1,j,k+1},		      & i=2,\\
\alpha_{i-1,j,k}-\left[1+\frac{1}{i(i-1)}\right]\alpha_{i-2,j,k}+\left[\frac{1}{3}+\frac{1}{i(i-2)}\right]\alpha_{i-3,j,k},		      & i\geq3.
                     \end{cases}
\end{equation}
\end{enumerate}
\end{proposition}
The analysis of this three-dimensional recursion is simplified considerably by the fact that for each $i,j$ the sequence $(\alpha_{i,j,k})_k$ becomes eventually constant. More precisely, we have the following result which may be compared to \cite[Lemma 2]{schlemm2012asymptotic} and can be proved along the same lines.
\begin{lemma}
\label{lemma-coeff-constant}
For each $i,j$ there exists a rational number $\beta_{i,j}$ such that $\alpha_{i,j,k}=\beta_{i,j}$ for all $k\geq i+j+1$. In particular, $\lim_{k\to\infty}{\alpha_{i,j,k}}=\beta_{i,j}$.
\end{lemma}
In \cref{tab-coeff-k1,tab-coeff-k2,tab-coeff-k3,tab-coeff-k4,tab-coeff-k5} the coefficients $\alpha_{i,j,k}$ with $k\geq i+j+1$ are printed in bold and the frontier coefficients $\alpha_{i,j,i+j+1}$ are marked with boxes. The assertion of \cref{lemma-coeff-constant} is thus easily appreciated by visual inspection. \Cref{eq-def-q-as-lim} in combination with \cref{lemma-coeff-constant} allows to recognise the limiting function $q$ as the generating function of the limiting coefficients $\beta_{i,j}$, i.\,e.\ $q(x,y)=\sum_{i,j}{\beta_{i,j}x^iy^j}\eqqcolon B_{\circ,\circ}(x,y)$. Before proceeding further we define convenient notation for the generating functions of the arrays $\alpha_{i,j,k}$ and $\beta_{i,j}$ along various dimensions and with various indices held fixed.

\begin{align*}
B_{\circ,j}(x) =& \sum_{i=0}^\infty{\beta_{i,j}x^i},\quad B_{i,\circ}(y) = \sum_{j=0}^\infty{\beta_{i,j}y^j},\\
A_{i,j,\circ}(z) =& \sum_{k=0}^\infty{\alpha_{i,j,k}z^k},\quad A_{\circ,j,\circ}(x,z) = \sum_{i,k\geq 0}{\alpha_{i,j,k}x^iz^k},\quad A_{i,\circ,\circ}(y,z) = \sum_{j,k\geq 0}{\alpha_{i,j,k}y^jz^k}.
\end{align*}
Here, the formal variables $x$, $y$ and $z$ correspond to indices $i$, $j$ and $k$, respectively, and the symbol $\circ$ indicates summation over the index that it replaces. The next result establishes how passing to the limit $k\to\infty$ can be accomplished at the level of generating functions.
\begin{lemma}
\label{lemma-Ggenred}
For each $x,y\in[0,1]$, it holds that $B_{\circ,\circ}(x,y)$ equals $\lim_{z\to 1^-}{(1-z)A_{\circ,\circ,\circ}(x,y,z)}$. Similarly, for non-negative integers $i,j$, it holds that
\begin{equation}
B_{i,\circ}(y)=\lim_{z\to 1^-}{(1-z)A_{i,\circ,\circ}(y,z)},\quad B_{\circ,j}(x)=\lim_{z\to 1^-}{(1-z)A_{\circ,j,\circ}(x,z)}.
\end{equation}
\end{lemma}
\begin{proof}
It suffices to prove the first claim, which follows directly from \cref{lemma-coeff-constant}; it allows us to write
\begin{equation*}
A_{\circ,\circ,\circ}(x,y,z) = \sum_{i,j}\sum_{k\leq i+j}{\alpha_{i,j,k}x^iy^jz^k} + \frac{z}{1-z}\sum_{i,j}{\beta_{i,j}(zx)^i(zy)^j}.
\end{equation*}
After multiplication by $(1-z)$ the first term vanishes as $z$ approaches one, whereas the second one converges to $B_{\circ,\circ}(x,y)$.
\end{proof}
We now begin analyzing the recursion for $\alpha_{i,j,k}$ in more detail. Throughout, we employ the powerful technique of generating functions as described in \cite{wilf2006generating}. In particular, we make use of the fact that a linear recursion equation for a sequence $c_n$ can be transformed into a differential equation for the generating function $x\mapsto\sum_n{c_n x^n}$ by multiplying the original recursion by $x^n$ and summing over $n$. Indeed, if the original recursion equation has polynomial coefficients, this property is shared by the resulting differential equation.
\begin{proposition}
\label{prop-Bcircj}
For positive integers $j$, the generating function $B_{\circ,j}$ is given by $B_{\circ,j}(x)=\beta_{1,j}G(x)$, where
\begin{equation}
\label{eq-GDef}
G(x) = \frac{9}{8}\left[F_{4,5}\left\{1/2\right\}F_{2,1}\left\{\frac{(1-x)^3}{2}\right\}-(1-x)^2F_{2,1}\left\{1/2\right\}F_{4,5}\left\{\frac{(1-x)^3}{2}\right\}\right].
\end{equation}
\end{proposition}
\begin{proof}
The last case of \cref{eq-recn5jgeq1} and \cref{lemma-coeff-constant} imply that, for $j\geq 1$, the sequence $(\beta_{i,j})_i$ satisfies the recursion
\begin{equation*}
\beta_{i,j} = \beta_{i-1,j}-\left[1+\frac{1}{i(i-1)}\right]\beta_{i-2,j}+\left[\frac{1}{3}+\frac{1}{i(i-2)}\right]\beta_{i-3,j},\quad i\geq 3.
\end{equation*}
After multiplying by $i(i-1)(i-2)x^i$ and summing over $i$ this translates into the differential equation
\begin{equation}
\label{eq-ODE-Bcircj}
4B_{\circ,j}(x) - 7(1-x) B_{\circ,j}'(x) + 3 (1-x)^2 B_{\circ,j}''(x) - \frac{1}{3}\left[2+(1-x)^3\right] B_{\circ,j}'''(x)=0
\end{equation}
for the generating functions $B_{\circ,j}(x)=\sum_{i=0}^\infty{\beta_{i,j}x^i}$. The initial conditions are
\begin{equation}
\label{eq-initialcons-Bcircj}
B_{\circ,j}(0)=\beta_{0,j}=0,\quad B_{\circ,j}'(0)=\beta_{1,j}, \quad \text{and}\quad B_{\circ,j}''(0) = 2\beta_{2,j} = \beta_{1,j},
\end{equation}
where the last equality follows from the second case of \cref{eq-recn5jgeq1}. 
% A basis for the space of solutions of \cref{eq-ODE-Bcircj} is given by
% \begin{equation*}
% \left\{F_{2,3}\left(\frac{1}{2}(1-x)^3\right), (1-x)^2 F_{4,5}\left(\frac{1}{2}(1-x)^3\right)\right\}
% \end{equation*}
The general solution of \cref{eq-ODE-Bcircj} is
\begin{align*}
	B_{\circ,j}(x)=&c_{0,j} (x-1) \, _3F_2\left\{1,\frac{1}{3}\left(3-\ii \sqrt{2}\right),\frac{1}{3}\left(3+\ii \sqrt{2}\right);\frac{2}{3},\frac{4}{3};\frac{1}{2} (x-1)^3\right\}\\
&+c_{1,j} (1-x)^2 F_{4,5}\left\{\frac{(1-x)^3}{2}\right\}+c_{2,j} F_{2,3}\left\{\frac{(1-x)^3}{2}\right\}.
\end{align*}
This can be obtained with the aid of a computer algebra system or checked using the power series representation of hypergeometric functions. The coefficients $c_{m,j}$, $m=1,2,3$, are determined by the initial conditions \labelcref{eq-initialcons-Bcircj} and are given by $c_{0,j} = 0$, $c_{1,j} = d_1\beta_{1,j}$, and $c_{2,j} = d_2\beta_{1,j}$, where $d_1$ and $d_2$ are explicit hypergeometric constants. Observing that $d_1+d_2$ equals $40/9$, which can be proved via Zeilberger's algorithm \cite{zeilberger1996AB}, the claim follows. 
\end{proof}
\Cref{prop-Bcircj} implies the decomposition
\begin{equation*}
B_{\circ,\circ}(x,y) = \sum_{j=1}^\infty{B_{\circ,j}(x)y^j} + B_{\circ,0}(x) = G(x)B_{1,\circ}(y)+B_{\circ,0}(x),
\end{equation*}
and it thus only remains to compute the two functions $B_{1,\circ}$ and $B_{\circ,0}$. Unfortunately, the corresponding cases of \cref{eq-recn5j0,eq-recn5jgeq1} can not be analyzed by simply passing to the limit $k\to\infty$ because they involve coefficients $\alpha_{i,j,k}$ with $j>k$. We therefore need to compute the two-dimensional generating function of the array $(\alpha_{i,j,k})_{i,k}$, a task which is directly modelled after the proof of \cref{prop-Bcircj}.
\begin{proposition}
\label{prop-G-circ-circ-j}
For positive integers $j$, the generating function $A_{\circ,j,\circ}\colon (x,z)\mapsto\sum_{i,k\geq 0}{\alpha_{i,j,k}x^iz^k}$ is given by $A_{\circ,j,\circ}(x,z)=A_{1,j,\circ}(z)G_2(x,z)$, where
\begin{equation}
\begin{split}
\label{eq-G2Def}
G_2(x,z)=&\frac{9}{2(3-z)^2}\left[F_{4,5}\left\{\frac{z}{z-3}\right\}F_{2,1}\left\{\frac{(1-x)^3 z}{3-z}\right\}\right.\\
	&\left.\qquad\qquad- (1-x)^2F_{1,2}\left\{\frac{z}{3-z}\right\}F_{4,5}\left\{\frac{(1-x)^3z}{3-z}\right\}\right].
\end{split}
\end{equation}
In particular, $G_2(x,1)=G(x)$.
\end{proposition}
\begin{proof}
The proof proceeds by multiplying the $i\geq 3$ case of \cref{eq-recn5jgeq1} by $i(i-1)(i-2)x^iz^k$ and summing over both $i$ and $k$. After some algebraic manipulations one arrives at a differential equation similar to \cref{eq-ODE-Bcircj} which can be solved using the same techniques.
\end{proof}
Our next task will be to transform the first case of the recursion relation \labelcref{eq-recn5jgeq1} into differential form. This corresponds to the case $j\geq 1$ and $i=1$. The proof proceeds as before by multiplying the recursion equation by $y^jz^k$, summing over $j$ and $k$ and interchanging the order of summation where necessary.
\begin{proposition}
\label{prop-ODE1-2var}
The generating functions $A_{1,\circ,\circ}$ and $A_{\circ,0,\circ}$ satisfy the integral equation
\begin{equation}
\begin{split}
\label{eq-ODE1-2var}
\frac1z\left(A_{1,\circ,\circ}(y,z)-1\right) =& G(y,z)\int_y^1{(A_{1,\circ,\circ}( w ,z)-1)\dd w } + \left(A_{1,\circ,\circ}(y,z)-1\right)\int_0^y{G( w ,z)\dd w }\\
  &+(1-y)A_{\circ,0,\circ}(y,z) + \int_0^y{A_{\circ,0,\circ}( w ,z)\dd w }.
\end{split}
\end{equation}
\end{proposition}
Setting $z=1$ in \cref{eq-ODE1-2var}, applying \cref{lemma-Ggenred} and differentiating once with respect to $x$ yields the following result which will be one of the main ingredients in the proof of our main theorem.
\begin{corollary}
\label{coro-ODE1-1var}
The generating functions $B_{1,\circ}$ and $B_{\circ,0}$ satisfy the integro-differential equation
\begin{align}
\label{eq-ODE1-1var}
B_{\circ,0}'(x)=\frac{1}{1-x}\left[-G'(x)\int_x^1{B_{1,\circ}(\xi)\dd\xi} + B_{1,\circ}'(x)\left(1-\int_0^x{G(\xi)\dd\xi}\right)\right].
\end{align}
\end{corollary}
A second differential equation relating $B_{1,\circ}$ and $B_{\circ,0}$ can be derived from \cref{eq-recn5j0}.
\begin{proposition}
The generating functions $B_{1,\circ}$ and $B_{\circ,0}$ satisfy the differential equation
\begin{align}
\label{eq-ODE2-1var}
&4(1-x) \int_0^x{B_{\circ,0}(\xi)\dd\xi} - (1-x)^2B_{\circ,0}(x) + \frac13\left[2+(1-x)^3\right]B_{\circ,0}'(x)\notag\\
&\quad = 3(1-x)\left(1-\int_0^x{G(\xi)\dd\xi}\right)B_{1,\circ}(x) + \frac13\left[2+(1-x)^3\right]G(x)B_{1,\circ}'(x).
\end{align}
\end{proposition}
\begin{proof}
Using the $i=1$ case of \cref{eq-recn5jgeq1} the two $\sum_{p=0}^{3k+1}$ sums can be eliminated from \cref{eq-recn5j0}. This allows to pass to the limit $k\to\infty$ and to replace the coefficients $\alpha_{i,j,k}$ by their steady-state values $\beta_{i,j}$ from \cref{lemma-coeff-constant}. Multiplying the resulting equation by $i(i-1)(i-2)x^i$ and summing over $i$ yields \cref{eq-ODE2-1var}. The details of the computation are lengthy and hence omitted.
\end{proof}
The two equations \labelcref{eq-ODE1-1var,eq-ODE2-1var}, together with the initial conditions that are tabulated in \cref{tab-coeff-k1,tab-coeff-k2,tab-coeff-k3,tab-coeff-k4,tab-coeff-k5}, are already enough to numerically compute the functions $B_{1,\circ}$ and $B_{\circ,0}$. It is, however, possible to simplify the problem even further to a single differential equation of order five.
\begin{proof}[Proof of \cref{thm-main}]
For better readability we introduce the antiderivatives
\begin{equation*}
	\mathcal{B}_0(x) = \int_0^{1-x}{B_{\circ,0}(\xi)\dd\xi},\quad \mathcal{B}_1(x) = \int_{1-x}^1{B_{1,\circ}(\xi)\dd\xi},\quad\text{and}\quad \mathcal{G}(x) = 1-\int_0^{1-x}{G(\xi)\dd\xi}.
\end{equation*}
A direct evaluation of the last antiderivative gives \cref{eq-defG}. After the substitution $x\to 1-y$, \cref{eq-ODE1-1var,eq-ODE2-1var} can be rewritten in terms of these functions as
\begin{equation}
\begin{aligned}
y\mathcal{B}_0''(y) =& -\mathcal{G}''(y)\mathcal{B}_1(y)+\mathcal{G}(y)\mathcal{B}_1''(y),\\
4y\mathcal{B}_0(y)+y^2\mathcal{B}_0'(y)+\frac13\left(2+y^3\right)\mathcal{B}_0''(y) =& 3y\mathcal{G}(y)\mathcal{B}_1'(y)-\frac13\left(2+y^3\right)\mathcal{G}'(y)\mathcal{B}_1''(y).
\end{aligned}
\end{equation}
Differentiating the latter equation three times eliminates the zeroeth and first derivative of $\mathcal{B}_0$; replacing the remaining second through fifth derivatives with expressions derived from the first equation of the last display produces the fifth-order differential equation $\sum_{i=0}^5{c_i(y)\mathcal{B}_1^{(j)}(y)}=0$. The coefficients $c_j$ are given by
\begin{align*}
\label{eq-coeffc}
c_0(y) =& 6 \left(5 y^3-2\right) \mathcal{G}''(y)+6y \left(5 y^3+2\right) \mathcal{G}^{(3)}(y)+y^2\left(9 y^3-6\right) \mathcal{G}^{(4)}(y)+y^3 \left(y^3+2\right) \mathcal{G}^{(5)}(y),\\
c_1(y) =& 3 y \left[\left(y^3+4\right) \mathcal{G}''(y)+y\left(3 y^3-4\right) \mathcal{G}^{(3)}(y)+y^2 \left(y^3+2\right) \mathcal{G}^{(4)}(y)\right],\\
c_2(y) =& 6 \left(2-5 y^3\right) \mathcal{G}(y)-6 y\left(13 y^3+2\right) \mathcal{G}'(y)-9 y^5 \mathcal{G}''(y)+y^3 \left[\left(11 y^3+4\right) \mathcal{G}^{(3)}(y)+\left(y^3+2\right) y \mathcal{G}^{(4)}(y)\right],\\
c_3(y) =& y \left[-3 \left(19 y^3+4\right) \mathcal{G}(y) + 3y \left(4-9 y^3\right) \mathcal{G}'(y)+4y^2 \left(4 y^3-1\right) \mathcal{G}''(y)+3 y^3 \left(y^3+2\right) \mathcal{G}^{(3)}(y)\right],\\
c_4(y) =& 3 y^2 \left[\left(2-6 y^3\right) \mathcal{G}(y) +2 y\left(y^3-1\right) \mathcal{G}'(y)+ y^2 \left(y^3+2\right) \mathcal{G}''(y)\right],\\
c_5(y) =& y^3 \left(y^3+2\right) \left(-\mathcal{G}(y)+y \mathcal{G}'(y)\right).
\end{align*}
It follows from \cref{prop-Bcircj} that the function $\mathcal{G}$ satisfies the differential equations
\begin{align*}
0 =&  3 y \mathcal{G}(y)+3 y^2 \mathcal{G}'(y)+\left(y^3+2\right) \mathcal{G}''(y), \\
0 =&  6 y^2 \mathcal{G}'(y)+\left(5 y^3-2\right) \mathcal{G}''(y)+y\left(y^3+2\right)\mathcal{G}^{(3)}(y), \\
0 =&  \left(11 y^3+4\right) \mathcal{G}''(y)+y\left(7 y^3-4\right) \mathcal{G}^{(3)}(y)+y^2\left(y^3+2\right) \mathcal{G}^{(4)}(y), \\
0 =& 18 y\left(11 y^3+16\right) \mathcal{G}^{(3)}(y) + 9 y^2\left(11 y^3-2\right) \mathcal{G}^{(4)}(y) + \left(11 y^6+26 y^3+8\right) \mathcal{G}^{(5)}(y). % \\
% 0 =&27 \left(11 y^3+40\right) y^2 \mathcal{G}^{(4)}(y)+\left(121 y^6+136 y^3-32\right) \mathcal{G}^{(5)}(y)+\left(11 y^6+38 y^3+32\right)y \mathcal{G}^{(6)}(y).
\end{align*}
Using these equations to eliminate the higher-order derivatives of $\mathcal{G}$ from the expressions for the coefficients $c_j(x)$ one obtains the claimed expressions \labelcref{eq-coeff-ODE}. Assertion \labelcref{eq-defBcirc0} is a reformulation of \cref{eq-ODE1-1var}.
\end{proof}

\bibliographystyle{plain}
%\bibliography{./../../RSM}

\begin{thebibliography}{20}

\bibitem{abramowitz1964handbook}
M.~Abramowitz and I.~A. Stegun.
\newblock {\em Handbook of mathematical functions with formulas, graphs, and
  mathematical tables}, volume~55 of {\em National Bureau of Standards Applied
  Mathematics Series}.
\newblock Courier Corporation, 1964R.

\bibitem{Bak1993punctuated}
P.~Bak and K.~Sneppen.
\newblock Punctuated equilibrium and criticality in a simple model of
  evolution.
\newblock {\em Phys. Rev. Lett.}, 71(24):4083, 1993.

\bibitem{bose2001bacterial}
I.~Bose and I.~Chaudhuri.
\newblock Bacterial evolution and {B}ak-{S}neppen model.
\newblock {\em Internat. J. Modern Phys. C}, 12(5):675--685, 2001.

\bibitem{de1994simple}
J.~De~Boer, B.~Derrida, H.~Flyvbjerg, A.~D. Jackson, and T.~Wettig.
\newblock Simple model of self-organized biological evolution.
\newblock {\em Phys. Rev. Lett.}, 73(6):906, 1994.

\bibitem{donangelo2002model}
R.~Donangelo and H.~Fort.
\newblock Model for mutation in bacterial populations.
\newblock {\em Phys. Rev. Let.}, 89(3):038101, 2002.

\bibitem{gould1972punctuated}
N.~Eldredge and S.~J. Gould.
\newblock Punctuated equilibria: an alternative to phyletic gradualism.
\newblock In T.~J.~M. Schopf, editor, {\em Models in Paleobiology}, chapter~5,
  pages 82--115. Freeman Copper, San Francisco, 1972.

\bibitem{flyvbjerg1993mean}
H.~Flyvbjerg, K.~Sneppen, and P.~Bak.
\newblock Mean field theory for a simple model of evolution.
\newblock {\em Phys. Rev. Lett.}, 71(24):4087, 1993.

\bibitem{gillet2006bounds}
A.~Gillett, R.~Meester, and M.~Nuyens.
\newblock Bounds for avalanche critical values of the {B}ak--{S}neppen model.
\newblock {\em Markov Process. Related Fields}, 12(4):679--694, 2006.

\bibitem{gillet2006maximal}
A.~Gillett, R.~Meester, and P.~Van Der~Wal.
\newblock Maximal avalanches in the {B}ak--{S}neppen model.
\newblock {\em J. Appl. Probab.}, 43(3):840--851, 2006.

\bibitem{grinfeld2011bak}
M.~Grinfeld, P.~A. Knight, and A.~R. Wade.
\newblock Bak--{S}neppen-type models and rank-driven processes.
\newblock {\em Phys. Rev. E}, 84(4):041124, 2011.

\bibitem{grinfeld2011rank}
M.~Grinfeld, Philip~A. Knight, and A.~R. Wade.
\newblock Rank-driven {M}arkov processes.
\newblock {\em J. Stat. Phys.}, 146(2):378--407, 2012.

\bibitem{jensen1988self}
H.~J. Jensen.
\newblock {\em Self-organized criticality: emergent complex behavior in
  physical and biological systems}, volume~10.
\newblock Cambridge University Press, 1988.

\bibitem{lenski2015LTEE}
R.~Lenski.
\newblock The {E}.\ coli long-term experimental evolution project site, 2015.
\newblock \url{http://myxo.css.msu.edu/ecoli}.

\bibitem{lenski1994dynamics}
Richard~E Lenski and Michael Travisano.
\newblock Dynamics of adaptation and diversification: a 10,000-generation
  experiment with bacterial populations.
\newblock {\em Proc. Natl. Acad. Sci.}, 91(15):6808--6814, 1994.

\bibitem{MeesterZnamenski2003limit}
R.~Meester and D.~Znamenski.
\newblock Limit behavior of the {B}ak--{S}neppen evolution model.
\newblock {\em Ann. Probab.}, 31(4):1986--2002, 2003.

\bibitem{MeesterZnamenski2004critical}
R.~Meester and D.~Znamenski.
\newblock Critical thresholds and the limit distribution in the
  {B}ak--{S}neppen model.
\newblock {\em Comm. Math. Phys.}, 246(1):63--86, 2004.

\bibitem{zeilberger1996AB}
M.~Petkov{\v{s}}ek, H.~S. Wilf, and D.~Zeilberger.
\newblock {\em {$A=B$}}.
\newblock A K Peters, Ltd., Wellesley, MA, 1996.

\bibitem{schlemm2012asymptotic}
E.~Schlemm.
\newblock Asymptotic fitness distribution in the {B}ak--{S}neppen model of
  biological evolution with four species.
\newblock {\em J. Stat. Phys.}, 148(2):191--203, 2012.

\bibitem{wilf2006generating}
H.~S. Wilf.
\newblock {\em generatingfunctionology}.
\newblock A K Peters, Ltd., Wellesley, MA, third edition, 2006.

\end{thebibliography}

\newpage
\appendix

\section{Coefficients}

In this appendix we collect the values of the coefficients $\alpha_{i,j,k}$ of the functions $q_k$ for $k=1,\ldots,5$. This serves as an illustration of \cref{lemma-coeff-constant} and to determine the initial conditions of the various differential equations encountered in the paper.

\begin{table}[ht]
\begin{subtable}[t]{.25\linewidth}\centering
\begin{tabular}{@{}c|ll@{}}
\toprule
\backslashbox{i}{j}  & 0   & 1 \\ \midrule
0 & 0   & 0 \\
1 & 1   & 0 \\
2 & -1  & 0 \\
3 & 1/3 & 0 \\
4 & 0   & 0 \\ \bottomrule
\end{tabular}
\caption{$k=1$}
\label{tab-coeff-k1}
\end{subtable}
\hspace{.1\linewidth}
\begin{subtable}[t]{.5\linewidth}\centering
\begin{tabular}{@{}c|llllll@{}}
\toprule
\backslashbox{i}{j} & 0     & 1    & 2    & 3    & 4    & 5 \\ \midrule
0   & 0     & 0    & 0    & 0    & 0    & 0 \\
1   & \boxed{\bf{0}}     & 1    & -3/2 & 1    & -1/4 & 0 \\
2   & 3     & -1/2 & 3/4  & -1/2 & 1/8  & 0 \\
3   & -19/3 & 0    & 0    & 0    & 0    & 0 \\
4   & 11/2  & 0    &      &      &      &   \\
5   & -9/4  & 0    &      &      &      &   \\
6   & 3/8   & 0    &      &      &      &   \\
7   & 0     & 0    &      &      &      &   \\ \bottomrule
\end{tabular}
\caption{$k=2$}
\label{tab-coeff-k2}
\end{subtable}
\caption{Coefficients $\alpha_{i,j,k}$ of $g_k$, $k=1,2$, as computed from \cref{eq-recgk} or \cref{eq-recn5jgeq1,eq-recn5j0}.}
\end{table}

\begin{table}[ht]
\begin{tabular}{@{}c|lllllllll@{}}
\toprule
\backslashbox{i}{j} & 0     & 1    & 2    & 3    & 4    & 5   & 6   & 7   & 8 \\ \midrule
0	& 0 & 0 	  & 0 		  & 0 		  & 0 		   & 0 		   & 0 		   & 0 		     & 0 \\
1	& \bf0 & \boxed{\bf1/5} & {29/10} & -{25/3} & {121/12} & -{32/5} & {32/15} & -{32/105} & 0 \\
2	& \boxed{\bf2/5} & {9/10} & -{59/20} & {31/6} & -{127/24} & {16/5} & -{16/15} & {16/105} & 0 \\
3	& {38/5} & -{5/3} & {5/2} & -{5/3} & {5/12} & 0 & 0 & 0 & 0 \\
4	& -{523/20} & 1 & -{3/2} & 1 & -{1/4} & 0 &  &  &  \\
5	& {75/2} & -{1/5} & {3/10} & -{1/5} & {1/20} & 0 &  &  &  \\
6 	& -{3551/120} & 0 & 0 & 0 & 0 & 0 &  &  &  \\
7	& {477/35} & 0 &  &  &  &  &  &  &  \\
8	& -{487/140} & 0 &  &  &  &  &  &  &  \\
9	& {487/1260} & 0 &  &  &  &  &  &  &  \\
10	& 0 & 0 &  &  &  &  &  &  &  \\ \bottomrule
\end{tabular}
\caption{Coefficients $\alpha_{i,j,3}$ of $g_3$ as computed from \cref{eq-recgk} or \cref{eq-recn5jgeq1,eq-recn5j0}.}
\label{tab-coeff-k3}
\end{table}

{\tiny
\begin{table}[ht]
\scalebox{.6}{
\begin{tabular}{@{}c|llllllllllll@{}}
\toprule
\backslashbox{i}{j} & 0     & 1    & 2    & 3    & 4    & 5   & 6   & 7   & 8   & 9   & 10   & 11 \\ \midrule
0	& 0 & 0 & 0 & 0 & 0 & 0 & 0 & 0 & 0 & 0 & 0 & 0 \\
1	& \bf0 & {\bf1/5} & \boxed{\bf1/2} & 7 & -{1879/60} & {8507/150} & -{10421/180} & {4589/126} & -{227/16} & {143/45} & -{143/450} & 0 \\
2	& {\bf2/5} & \boxed{\bf1/10} & {53/20} & -{71/6} & {3089/120} & -{10427/300} & {11189/360} & -{23329/1260} & {227/32} & -{143/90} & {143/900} & 0 \\
3 	& \boxed{\bf1} & {2/3} & -{19/3} & {134/9} & -{307/18} & {32/3} & -{32/9} & {32/63} & 0 & 0 & 0 & 0 \\
4	& {1019/60} & -{51/20} & {281/40} & -{133/12} & {517/48} & -{32/5} & {32/15} & -{32/105} & 0 & 0 & 0 & 0 \\
5	& -{2591/30} & {311/100} & -{1061/200} & {289/60} & -{673/240} & {32/25} & -{32/75} & {32/525} & 0 & 0 & 0 & 0 \\
6	& {39877/225} & -{223/120} & {223/80} & -{223/120} & {223/480} & 0 & 0 & 0 & 0 & 0 & 0 & 0 \\
7	& -{438271/2100} & {17/30} & -{17/20} & {17/30} & -{17/120} & 0 & 0 & 0 & 0 & 0 & 0 & 0 \\
8	& {15035/96} & -{17/240} & {17/160} & -{17/240} & {17/960} & 0 & 0 & 0 & 0 & 0 & 0 & 0 \\
9	& -{55459/720} & 0 & 0 & 0 & 0 & 0 & 0 & 0 & 0 & 0 & 0 & 0 \\
10	& {1224179/50400} & 0 & 0 & 0 & 0 & 0 & 0 & 0 & 0 & 0 & 0 & 0 \\
11	& -{113287/25200} & 0 & 0 & 0 & 0 & 0 & 0 & 0 & 0 & 0 & 0 & 0 \\
12	& {113287/302400} & 0 & 0 & 0 & 0 & 0 & 0 & 0 & 0 & 0 & 0 & 0 \\
13	& 0 & 0 & 0 & 0 & 0 & 0 & 0 & 0 & 0 & 0 & 0 & 0 \\ \bottomrule
\end{tabular}
}
\caption{Coefficients $\alpha_{i,j,4}$ of $g_4$ as computed from \cref{eq-recgk} or \cref{eq-recn5jgeq1,eq-recn5j0}.}
\label{tab-coeff-k4}
\end{table}
}

\begin{table}[ht]
\scalebox{.4}{
\begin{tabular}{@{}c|lllllllllllllll@{}}
\toprule
\backslashbox{i}{j} & 0     & 1    & 2    & 3    & 4    & 5   & 6   & 7   & 8   & 9   & 10   & 11   & 12   & 13   & 14\\ \midrule
0	& 0 & 0 & 0 & 0 & 0 & 0 & 0 & 0 & 0 & 0 & 0 & 0 & 0 & 0 & 0 \\
1	& \bf0 & {\bf1/5} & {\bf1/2} & \boxed{\bf3/5} & {319/20} & -{4889/50} & {71923/300} & -{357143/1050} & {879181/2800} & -{62029/315} & {53009/630} & -{12424/525} & {896/225} & -{896/2925} & 0 \\
2	& {\bf2/5} & {\bf1/10} & \boxed{\bf1/4} & {67/10} & -{943/24} & {31681/300} & -{319979/1800} & {1300879/6300} & -{958631/5600} & {64031/630} & -{267047/6300} & {6212/525} & -{448/225} & {448/2925} & 0 \\
3	& \bf1 & \boxed{\bf-2/15} & {31/15} & -20 & {1121/18} & -{9083/90} & {53257/540} & -{115301/1890} & {1135/48} & -{143/27} & {143/270} & 0 & 0 & 0 & 0 \\
4	& \boxed{\bf47/60} & {13/20} & -{359/40} & {371/12} & -{14231/240} & {11147/150} & -{11477/180} & {23473/630} & -{227/16} & {143/45} & -{143/450} & 0 & 0 & 0 & 0 \\
5	& {5651/150} & -{321/100} & {2947/200} & -{629/20} & {46771/1200} & -{23627/750} & {16469/900} & -{25969/3150} & {227/80} & -{143/225} & {143/2250} & 0 & 0 & 0 & 0 \\
6	& -{37733/150} & {1199/200} & -{17927/1200} & {7867/360} & -{5855/288} & {892/75} & -{892/225} & {892/1575} & 0 & 0 & 0 & 0 & 0 & 0 & 0 \\
7	& {214099/315} & -{12527/2100} & {45197/4200} & -{13609/1260} & {36457/5040} & -{272/75} & {272/225} & -{272/1575} & 0 & 0 & 0 & 0 & 0 & 0 & 0 \\
8	& -{2253961/2100} & {9967/2800} & -{93511/16800} & {20987/5040} & -{32411/20160} & {34/75} & -{34/225} & {34/1575} & 0 & 0 & 0 & 0 & 0 & 0 & 0 \\
9	& {33882311/30240} & -{1957/1512} & {1957/1008} & -{1957/1512} & {1957/6048} & 0 & 0 & 0 & 0 & 0 & 0 & 0 & 0 & 0 & 0 \\
10	& -{20505517/25200} & {289/1080} & -{289/720} & {289/1080} & -{289/4320} & 0 & 0 & 0 & 0 & 0 & 0 & 0 & 0 & 0 & 0 \\
11	& {694474463/1663200} & -{289/11880} & {289/7920} & -{289/11880} & {289/47520} & 0 & 0 & 0 & 0 & 0 & 0 & 0 & 0 & 0 & 0 \\
12	& -{497946013/3326400} & 0 & 0 & 0 & 0 & 0 & 0 & 0 & 0 & 0 & 0 & 0 & 0 & 0 & 0 \\
13	& {55461661/1544400} & 0 & 0 & 0 & 0 & 0 & 0 & 0 & 0 & 0 & 0 & 0 & 0 & 0 & 0 \\
14	& -{6257393/1201200} & 0 & 0 & 0 & 0 & 0 & 0 & 0 & 0 & 0 & 0 & 0 & 0 & 0 & 0 \\
15	& {6257393/18018000} & 0 & 0 & 0 & 0 & 0 & 0 & 0 & 0 & 0 & 0 & 0 & 0 & 0 & 0 \\
16	& 0 & 0 & 0 & 0 & 0 & 0 & 0 & 0 & 0 & 0 & 0 & 0 & 0 & 0 & 0 \\ \bottomrule
\end{tabular}
}
\caption{Coefficients $\alpha_{i,j,5}$ of $g_5$ as computed from \cref{eq-recgk} or \cref{eq-recn5jgeq1,eq-recn5j0}.}
\label{tab-coeff-k5}

\end{table}

\end{document}